\documentclass[11pt,reqno]{amsart}
\input{glyphtounicode}
\usepackage[T1]{fontenc}
\usepackage[utf8]{inputenc}
\usepackage{lmodern}
\usepackage[a4paper,margin=27mm]{geometry}
\usepackage{amsmath,amssymb,amsthm,mathtools}
\usepackage{microtype,booktabs,array,enumitem}
\usepackage{needspace}
\usepackage{xcolor}
\definecolor{linkblue}{RGB}{25,60,105}
\usepackage[colorlinks=true,linkcolor=linkblue,citecolor=linkblue,urlcolor=linkblue]{hyperref}
\usepackage{bookmark}
\usepackage[nameinlink,noabbrev]{cleveref}
\crefname{theorem}{Theorem}{Theorems}
\crefname{lemma}{Lemma}{Lemmas}
\crefname{proposition}{Proposition}{Propositions}
\crefname{corollary}{Corollary}{Corollaries}
\crefname{definition}{Definition}{Definitions}
\crefname{section}{Section}{Sections}
\crefname{subsection}{Section}{Sections}
\crefname{problem}{Problem}{Problems}
\crefname{question}{Question}{Questions}
\allowdisplaybreaks[2]
\setlist{itemsep=3pt,topsep=5pt}
\numberwithin{equation}{section}
\newtheorem{theorem}{Theorem}[section]
\newtheorem{lemma}[theorem]{Lemma}
\newtheorem{proposition}[theorem]{Proposition}
\newtheorem{corollary}[theorem]{Corollary}
\theoremstyle{definition}
\newtheorem{definition}[theorem]{Definition}
\newtheorem{problem}[theorem]{Problem}

\theoremstyle{remark}

\newcommand{\dH}{d_{\mathrm H}}

\title[Chromatic villainy: unions and algorithms]{Disjoint unions and fixed-parameter algorithms\\for chromatic villainy}
\author[Jiang, Jiang, Li, Wang, and Zhang]{Jiang Jiang, Yuhan Jiang, Yiyang Li, Lujia Wang, and Aitao Zhang}
\address{Hainan Bielefeld University of Applied Sciences, Danzhou, China}
\email{lujia.wang@hibiuh.edu.cn}
\email{\{x20250083,x20250116,x20250157,x20250275\}@stu.hibiuh.edu.cn}
\hypersetup{pdftitle={Disjoint unions and fixed-parameter algorithms for chromatic villainy},
pdfauthor={Jiang Jiang, Yuhan Jiang, Yiyang Li, Lujia Wang, Aitao Zhang}}
\date{7 October 2026}
\begin{document}
\begin{abstract}
Chromatic villainy is the maximum, over rearrangements of optimal
graph colorings, of the minimum Hamming distance to a proper coloring
with the same color multiplicities. Weak villainy allows the repair to
change these multiplicities. We give exact composition laws for
disjoint unions and determine the eventual growth under repetition.
Write $mG$ for the disjoint union of $m$ copies of $G$.
For every graph $G$, with $k=\chi(G)$, weak villainy satisfies
$B_{\mathrm w}(mG)=mF_k(G)$ whenever $k\mid m$, where $F_k(G)$ is
the maximum distance to a proper $k$-coloring over all assignments of
$k$ color labels to vertices. Its bounded defect from $mF_k(G)$ is eventually
periodic; we give an explicit stabilization bound and identify the
least period, which divides $k$. For ordinary villainy, we obtain a
computable rational slope $\Lambda(G)\ge F_k(G)$ such that
$B(mG)=m\Lambda(G)+O_G(1)$, where the error is bounded independently
of $m$ and eventually periodic, with a computable period and threshold.
Consequently the difference between the two
values is bounded exactly when $\Lambda(G)=F_k(G)$, and otherwise
grows linearly. The slope is determined by a finite optimization using
at most $k$ source assignments. We also answer the unrestricted
superadditivity questions negatively: for every $C>0$, a connected
graph can lose more than $C$ in weak villainy after union with a
triangle, and some connected $3$-chromatic graph $G$ satisfies
$B(G\sqcup G)<2B(G)-C$. Finally, both exact values on an $n$-vertex
graph can be computed in time $f(B_{\mathrm w}(G))(n+1)^{O(1)}$,
without either value being supplied. Threshold recognition is
fixed-parameter tractable as well.
\end{abstract}
\maketitle
\section{Introduction}

Chromatic villainy measures the cost of repairing a rearrangement of
an optimal coloring. The cost counts vertices whose colors change.
In the \emph{villainy} $B(G)$ of Clark et al.~\cite{clark2006},
the repair must preserve every color multiplicity; \emph{weak
villainy} $B_{\mathrm w}(G)$ permits different multiplicities on the
same $\chi(G)$-color palette. Both parameters maximize the minimum
repair distance over all such rearrangements. Thus
$0\le B_{\mathrm w}(G)\le B(G)\le |V(G)|$.
A \emph{word} is any assignment of color labels to vertices, without
requiring properness. Its \emph{profile} is the vector of its labeled color
multiplicities; it is feasible if a proper coloring realizes it
(\cref{def:profile}). We call a rearranged source coloring a disturbance.

We study the behavior of these parameters under disjoint union and the
complexity of recognizing graphs with a fixed value. Both questions
were raised in the early literature; see~\cite{clark2006,west2006} and
the account of Jahanbekam and Lin~\cite{jahanbekam2021}. Clark et al.
proved that weak villainy is superadditive under disjoint union when
the summands have the same chromatic number. They proved the ordinary
analogue under the additional assumption that every optimal coloring
of one summand has equal-sized color classes. It was natural to ask
whether these restrictions could be removed, and to determine when
equality holds. On the algorithmic side, the classifications of graphs
with villainy zero~\cite{clark2006} and villainy
two~\cite{jahanbekam2021} give polynomial-time recognition in those
cases. The complexity at other fixed values remained to be determined.

This paper continues our study in~\cite{jiang2026exact}, where we
determined odd-cycle profile radii, obtained complete multipartite
formulas, and determined the extremal values at fixed order and
chromatic number. Here we turn to disjoint unions and computation.
The profile viewpoint remains central, but composition requires
tracking transfers of color counts and, for ordinary repair,
simultaneous responses to different target profiles.
Our aim is to describe how global multiplicity constraints combine
local repair costs, and to determine when those costs can be computed
efficiently. For unions, we retain source profiles and \emph{repair
responses}, which record the cost of repairing one source at each
target profile (\cref{def:responses}). For graphs of small weak
villainy, a structural bound reduces all relevant colorings to boundedly
many vertex types, grouping vertices with the same neighbors in a
specified vertex set (\cref{rec:types}).

\subsection{Disjoint unions and profile composition}

The requirement that the global source profile be feasible couples
the components.
Although a disturbance preserves the total number of occurrences of
each color, its restriction to a component need not have the profile
of any proper coloring of that component. Similarly, an ordinary
repair may transfer color multiplicities between components while
preserving their global totals. When the chromatic numbers differ,
one component also gains access to a larger palette. These features
distinguish villainy from the unrestricted Hamming covering radius,
the greatest distance of any word to a prescribed set of words.
That radius is additive on Cartesian products of codes
\cite[Section~II.C]{cohen1985}.

Our first results answer both unrestricted superadditivity questions
negatively. For every real $C>0$,
there exists a connected graph $G$ such that
\[
 B_{\mathrm w}(G\sqcup K_3)
 <B_{\mathrm w}(G)+B_{\mathrm w}(K_3)-C.
\]
Thus one summand in the weak counterexample can be a fixed triangle.
For ordinary villainy, equal chromatic numbers do not suffice: for
every real $C>0$, there exists a connected graph $G$ with $\chi(G)=3$
such that
\[
 B(G\sqcup G)<2B(G)-C.
\]
\cref{union:weak-counterexample,union:strong-counterexample} give explicit constructions.

To describe what replaces additivity, we keep track of source and
target profiles. For weak repair, the relevant local information
is the largest repair distance among words of each prescribed source
profile, with the palette fixed in advance.
\cref{union:transfer-theorem}
expresses the value on a union through these local radii. The formula
separates the effect of changing the palette from that of transferring
color counts between components. It gives an exact equality criterion
in local profile data, including when the chromatic numbers differ.

Repeated copies admit a particularly simple weak-repair theorem. Let
$mG$ denote the disjoint union of $m$ copies of $G$, put $k=\chi(G)$,
and let $F_k(G)$ be the $k$-fixing number of Garnero et al.
\cite{garnero2018}: the maximum distance to a proper $k$-coloring over
\emph{all} source words on that palette. If $m=qk+r$ with $0\le r<k$,
\cref{union:repetition-theorem} proves
\[
 0\le mF_k(G)-B_{\mathrm w}(mG)
 \le r\bigl(F_k(G)-B_{\mathrm w}(G)\bigr).
\]
In particular, $B_{\mathrm w}(mG)=mF_k(G)$ whenever $k$ divides $m$,
and $B_{\mathrm w}(mG)/m$ converges to $F_k(G)$. Moreover, weak
villainy is additive over every repetition if and only if
$B_{\mathrm w}(G)=F_k(G)$; it suffices to test the union of $k$ copies.
The proof balances the color counts of a worst unrestricted source
by cyclically relabeling its colors on $k$ copies.
\Cref{union:repetition-stabilization} makes the eventual behavior
effective. For $n=|V(G)|$ and $k\ge2$, put
\[
 M=\bigl(2(k-1)n+1\bigr)^{k-1},\qquad N=(k-1)M.
\]
For every $m\ge N$, the defect $mF_k(G)-B_{\mathrm w}(mG)$ is
periodic. Its least period is
\[
 h=\gcd\{\ell\ge1:B_{\mathrm w}(\ell G)=\ell F_k(G)\},
 \qquad h\mid k,
\]
and the gcd can be computed using only $\ell\le M$. Beyond $N$,
the fixing-number bound is attained exactly when $h\mid m$.
These weak-repair results hold on any fixed palette of size at least
$\chi(G)$.

Ordinary repair requires more information. A single disturbed coloring
has a repair cost for each possible target profile, and these costs
must be retained together. Maximizing each coordinate separately can
combine costs that no single disturbance realizes. We therefore use
the full response functions and discard only those bounded above at
every target profile by another realizable response. To combine two
responses, we minimize the sum of their costs over all splits of the
target profile. This operation, called min-plus convolution in
\cref{def:responses}, gives \cref{union:ordinary-composition}.
The resulting equality tests are finite criteria in component data;
they do not constitute a structural classification of all equality
pairs.

The response description also yields a general repetition theorem
for ordinary villainy. \Cref{union:ordinary-repetition} gives a
computable rational number $\Lambda(G)\ge F_k(G)$ such that
\[
 B(mG)=m\Lambda(G)+O_G(1).
\]
Here $O_G(1)$ denotes an error bounded by a constant depending only on
$G$. The slope comes from an optimization that allows fractional
proportions of copies to receive each source and target coloring,
as defined in~\eqref{union:fractional-response}.
Its maximum is attained by a distribution using at most $k$ source words.
Moreover, $B(mG)-m\Lambda(G)$ is eventually periodic, with a
computable period and threshold. Thus \cref{union:repetition-gap}
determines the possible growth of the difference between the repair
rules: it is bounded precisely when $\Lambda(G)=F_k(G)$, in which
case it is eventually periodic; otherwise it grows linearly.

\subsection{Fixed-parameter recognition}

Our algorithmic result computes both exact parameters using the
potentially smaller weak villainy as parameter. \Cref{rec:main}
gives running time $h(B_{\mathrm w}(G))(n+1)^c$, where $h$ is
computable and $c$ is an absolute constant, without either value being
supplied. Given an integer $t\ge0$, it also computes
the two truncated values
\[
 \min\{B(G),t+1\}\qquad\text{and}\qquad
 \min\{B_{\mathrm w}(G),t+1\}
\]
in time $f(t)(n+1)^c$ for a computable function $f$.
Thus both equality recognition, deciding whether a parameter equals
$t$, and threshold recognition, deciding whether it is at most $t$,
are fixed-parameter tractable in $t$: the polynomial exponent in the
input size is independent of $t$. They are polynomial-time problems
at every fixed value.
In particular,
unless $\mathrm P=\mathrm{NP}$, there is no fixed value at which
villainy recognition is NP-complete.

The structural input is a matching bound. If
$B_{\mathrm w}(G)\le t$, then either $G$ has a vertex cover of size at
most $2t$, or $\chi(G)\ge n-2t$ and the complement of $G$ has a vertex
cover of size at most $4t$. The same conclusion holds when
$B(G)\le t$. The proof groups vertices outside these covers by their
neighborhoods and retains only source profiles realized by optimal
proper colorings.
In the small-cover case, a bounded-length formula with variables for
the type sizes permits elimination of its quantified variables before
the input multiplicities are substituted. In the complement-cover
case, symmetries that fix the source coloring remove the need to
enumerate singleton color names individually.
The polynomial exponent is independent of $t$; the proof does not
give a practical bound on the function $f$.

\subsection{Related repair problems and organization}

Garnero et al.~\cite{garnero2018} study repair of an arbitrary given
coloring and the fixing number used above. De Biasi and
Lauri~\cite{debiasi2019} consider a source promised to be a
rearrangement of an optimal coloring. Their $r$-Fix-Promise problem
is an individual weak-repair instance; their swap problem uses the
number of color exchanges as cost. Our recognition problem
quantifies over every admissible disturbance.

Count-aware component optimization also appears in the two-color
repair algorithm of Fellows et al.~\cite{fellows2026}. It chooses
orientations of bipartite components for one given source, subject
to total color counts and a swap budget. For villainy, one must additionally retain jointly realizable
responses before taking the worst-source maximum.

Finally, endpoint Hamming distance differs from proper-coloring
reconfiguration~\cite{Bousquet2024,fuchs2025}, where every intermediate
coloring must be proper. Ordinary villainy constrains only the initial
and final color multiplicities.

\cref{sec:preliminaries} fixes the notation and definitions.
\cref{union:section} gives the counterexamples to
superadditivity. \cref{union:composition-section} develops the
profile composition laws and the repeated-copy theorem.
\cref{rec:section} proves fixed-parameter recognition, and
\cref{sec:questions} discusses further questions.

\section{Color profiles and repair distances}\label{sec:preliminaries}

All graphs are finite and simple. Unless otherwise stated, a graph $G$ is
nonempty, has order $n$, and has chromatic number $k$. We use a fixed set
$\mathcal C$ of $k$ labeled colors.

\Needspace{6\baselineskip}
\begin{definition}[Words and profiles]\label{def:profile}
A \emph{word} on $G$ is a map
$w:V(G)\to\mathcal C$, with no properness assumption. Its \emph{profile}
is the vector $\mathbf m(w)=(m_X(w))_{X\in\mathcal C}$, where
$m_X(w)=|w^{-1}(X)|$. Bold letters denote vectors. Inequalities between
vectors are coordinatewise, $\mathbf e_X$ is a coordinate unit vector,
and $\mathbf1$ is an all-one vector of the required dimension.

Write $\mathcal P(G)$ for the proper colorings with palette $\mathcal C$.
Since $|\mathcal C|=\chi(G)$, each such coloring uses every color. A profile
$\mathbf m$ is \emph{feasible} if it occurs in $\mathcal P(G)$.
\end{definition}

For a feasible profile, put
\[
 \mathcal W_{\mathbf m}(G)=
 \{w:V(G)\to\mathcal C:\mathbf m(w)=\mathbf m\},\qquad
 \mathcal P_{\mathbf m}(G)=\mathcal W_{\mathbf m}(G)\cap\mathcal P(G).
\]
We omit $G$ from $\mathcal W_{\mathbf m}(G)$ when the vertex set is clear.
Every word in $\mathcal W_{\mathbf m}(G)$ is a rearrangement of every
proper coloring with that profile. We call such a word an
\emph{admissible disturbance}; equivalently, its profile is feasible.
The Hamming distance and the number
of agreements are
\[
 d_{\mathrm H}(w,y)=|\{v:w(v)\ne y(v)\}|,
 \qquad \operatorname{agr}(w,y)=n-d_{\mathrm H}(w,y).
\]
The cost therefore counts changed vertices; exchanging two different
colors on two vertices has cost two.

\begin{definition}[Repair radii and villainy]\label{def:repair-radii}
The repair radius of a profile and its weak counterpart are
\begin{align*}
 R_G(\mathbf m)&=
 \max_{w\in\mathcal W_{\mathbf m}(G)}
 \min_{y\in\mathcal P_{\mathbf m}(G)}d_{\mathrm H}(w,y),\\
 R^{\mathrm w}_G(\mathbf m)&=
 \max_{w\in\mathcal W_{\mathbf m}(G)}
 \min_{y\in\mathcal P(G)}d_{\mathrm H}(w,y).
\end{align*}
With the terminology of Clark et al.~\cite{clark2006},
\[
 B(G)=\max_{\mathbf m\text{ feasible}}R_G(\mathbf m),\qquad
 B_{\mathrm w}(G)=\max_{\mathbf m\text{ feasible}}R^{\mathrm w}_G(\mathbf m).
\]
Both parameters are zero on the empty graph by convention.
\end{definition}

The inclusion
of the two target sets gives $B_{\mathrm w}(G)\le B(G)$. In particular,
the weak target may change the profile, but it still uses the same palette
of $\chi(G)$ colors. A simultaneous permutation of the color labels in
the source and target preserves distances and all profile conditions.

We write $[k]=\{1,\ldots,k\}$. For sets of vectors $A,B$, their sum is
$A+B=\{\mathbf a+\mathbf b:\mathbf a\in A,\mathbf b\in B\}$.
The graph $mG$ is the disjoint union of $m$ vertex-disjoint copies of $G$.

\section{Counterexamples to superadditivity}\label{union:section}

Clark et al.~\cite[Propositions~3.7 and~3.8]{clark2006} proved
\begin{equation}\label{union:known-weak}
 B_{\mathrm w}(G\sqcup H)\ge B_{\mathrm w}(G)+B_{\mathrm w}(H)
 \qquad\text{when }\chi(G)=\chi(H),
\end{equation}
and its ordinary analogue when, in addition, every optimal coloring
of one summand has equal-sized color classes. We show that the
unrestricted inequalities fail by an arbitrarily large amount, even
for connected summands.

\Needspace{6\baselineskip}
\begin{theorem}\label{union:weak-counterexample}
For every real $C>0$, there exists a connected graph $G$ such that
\[
 B_{\mathrm w}(G\sqcup K_3)
 <B_{\mathrm w}(G)+B_{\mathrm w}(K_3)-C.
\]
\end{theorem}
\begin{proof}
For a positive integer $s$, let $T_s$ consist of an edge $uv$ with
$3s$ leaves adjoined to each endpoint. Both bipartition classes have
size $3s+1$. The connected bipartite formula
\cite[Proposition~3.3]{clark2006} gives
\[
 B_{\mathrm w}(T_s)=2\left\lfloor\frac{3s+1}{2}\right\rfloor,
 \qquad B_{\mathrm w}(K_3)=0.
\]
We show that $B_{\mathrm w}(T_s\sqcup K_3)\le2s+3$.

The union has chromatic number three. For any three-color source $w$,
let $N_u(i),N_v(i)$ count the leaves of color $i$ at the two centers.
Repairing the center colors to distinct labels $i,j$ has minimum tree
cost
\[
 N_u(i)+N_v(j)
 +\mathbf1_{\{w(u)\ne i\}}+\mathbf1_{\{w(v)\ne j\}}.
\]
Its average over the six ordered pairs is $2s+4/3$. Some integer cost
is therefore at most $2s+1$. The triangle needs at most two changes,
giving the required upper bound.

It follows that
\[
 B_{\mathrm w}(T_s)+B_{\mathrm w}(K_3)
 -B_{\mathrm w}(T_s\sqcup K_3)
 \ge 2\left\lfloor\frac{3s+1}{2}\right\rfloor-(2s+3)
 \ge s-3.
\]
Choosing an integer $s>C+3$ proves the stated strict inequality.
\end{proof}

For the ordinary parameter, the chromatic numbers can be equal and
the two connected summands can even be isomorphic. We use the following
general observation about a feasible profile $\mathbf p$ on $N$ vertices.
If two proper targets of that profile have disjoint color-$X$ classes
of size $p_X$, then every source of profile $\mathbf p$ agrees with one
of them on at least
\begin{equation}\label{union:disjoint-class-bound}
 \max\left\{0,\left\lceil\frac{3p_X-N}{2}\right\rceil\right\}
\end{equation}
vertices. Indeed, the source $X$-set has size $p_X$, and its intersection
with the union of the two target $X$-sets has size at least $3p_X-N$.
One of the two intersections has at least half that size.

\begin{theorem}\label{union:strong-counterexample}
For every real $C>0$, there exists a connected graph $G$ with
$\chi(G)=3$ such that
\[
 B(G\sqcup G)<2B(G)-C.
\]
Here the union consists of two vertex-disjoint copies of $G$.
\end{theorem}
\begin{proof}
For each positive integer $t$, take
\[
 J_t=\overline K_{5t}\vee
       \bigl(K_{t,2t}\sqcup K_{t,3t}\bigr),
\]
where the join adds every edge between its two arguments.
This is a connected graph of order $12t$ and chromatic number three.
Denote its five independent vertex sets by
$S,L_1,R_1,L_2,R_2$, of respective sizes $5t,t,2t,t,3t$.
The set $S$ is joined to all four other sets, and the remaining edges
are those between $L_i$ and $R_i$ for $i=1,2$.

In a proper three-coloring, the bipartite remainder requires two colors,
so $S$ must use only the third. Each of its two connected components
then has its two possible binary colorings. The possible unordered
profiles on $J_t$ are precisely $t(2,5,5)$ and $t(3,4,5)$.
The labeled profile $(3t,4t,5t)$ has a unique proper realization:
$C$ occupies $S$, $A$ occupies $R_1\cup L_2$, and $B$ occupies
$L_1\cup R_2$. A source with this profile and no agreement is described
by the matrix
\[
 t\begin{pmatrix}0&2&1\\0&0&4\\3&2&0\end{pmatrix},
\]
whose rows are these target classes and whose columns count the source
colors. Its row and column sums are $(3t,4t,5t)$, and its diagonal
vanishes. Hence $B(J_t)=12t$.

Adding two permutations of the possible profiles gives exactly the
following sorted profiles on $2J_t$, divided by $t$:
\[
 (4,10,10),\ (5,9,10),\ (6,8,10),\ (6,9,9),\
 (7,7,10),\ (7,8,9),\ (8,8,8).
\]
For clarity, adding two copies of $(2,5,5)$ gives the first or fifth;
mixing the two types gives the second, third, or sixth; adding two
copies of $(3,4,5)$ gives the third through seventh.
In every unbalanced case the table supplies two targets with disjoint
$C$-classes. Each five-letter word assigns colors to
$(S,L_1,R_1,L_2,R_2)$ in one copy, and the separator distinguishes
the two copies.
\begin{center}
\begin{tabular}{c@{\qquad}c@{\qquad}c}
\toprule
Profile divided by $t$ & First target & Second target\\
\midrule
$(4,10,10)$ & $BACAC\mid BACAC$ & $CABAB\mid CABAB$\\
$(5,9,10)$ & $BACAC\mid CBAAB$ & $CBAAB\mid BACAC$\\
$(6,8,10)$ & $BACAC\mid CABBA$ & $CABBA\mid BACAC$\\
$(6,9,9)$ & $BCAAC\mid CBAAB$ & $CBAAB\mid BCAAC$\\
$(7,7,10)$ & $ABCBC\mid BACAC$ & $CABAB\mid CBABA$\\
$(7,8,9)$ & $ACBBC\mid CABAB$ & $CABAB\mid ACBBC$\\
\bottomrule
\end{tabular}
\end{center}
Every word is proper: its first letter differs from the other four,
its second differs from its third, and its fourth differs from its
fifth. Weighted letter counts verify each profile. In no position
do both targets have $C$, so their $C$-classes are disjoint.

Here $N=24t$ and $p_C\ge9t$. By
\eqref{union:disjoint-class-bound}, every source of an unbalanced
profile has a repair retaining at least $\lceil3t/2\rceil$ vertices.
For the balanced profile $(8t,8t,8t)$, the three cyclic relabelings of
any proper target still have that profile and together agree once
at every vertex. Some target therefore retains at least $8t$ vertices.
Consequently
\[
 B(2J_t)\le24t-\left\lceil\frac{3t}{2}\right\rceil,
 \qquad
 2B(J_t)-B(2J_t)\ge\left\lceil\frac{3t}{2}\right\rceil.
\]
Choose an integer $t>2C/3$. Then $\lceil3t/2\rceil>C$, which gives
the stated strict inequality. The upper bound on $B(2J_t)$ is all
that is needed here; no exact value for it is asserted.
\end{proof}

\section{Profile composition and repeated unions}\label{union:composition-section}

The counterexamples in \cref{union:section} show the effect of permitting
color counts to move between components. We now retain the local profile
data and derive exact composition laws. For repeated copies, these laws
give eventual linear growth with periodic corrections under both repair
rules. Weak repair has the explicit slope $F_k(G)$; ordinary repair has
a rational slope described by a finite fractional optimization.

\subsection{Weak repair and transfer of color counts}
\label{union:weak-general}

To analyze arbitrary summands, their colorings must first be expressed
in a common palette.

\begin{definition}[Common-palette radii]\label{def:common-palette}
Fix $k\ge\chi(J)$, and let $\mathcal P_k(J)$ be
the set of proper maps from $V(J)$ to $[k]$.  These maps may leave some
labels unused.  Write
\[
 \mathcal A_k(J)=
 \left\{\mathbf a\in\mathbb Z_{\ge0}^k:
          \sum_{i=1}^k a_i=|V(J)|\right\},
 \qquad
 \mathcal M_k(J)=\{\mathbf m(y):y\in\mathcal P_k(J)\}.
\]
Thus $\mathcal A_k(J)$ contains all source profiles, whereas
$\mathcal M_k(J)$ contains the profiles realized by proper colorings.
The distinction matters even when $k=\chi(J)$: a source restricted to
one summand of a union need not have a feasible profile on that summand.

For every $\mathbf a\in\mathcal A_k(J)$, define
\[
 \rho_{J,k}(\mathbf a)=
 \max_{\substack{w:V(J)\to[k]\\\mathbf m(w)=\mathbf a}}
 \min_{y\in\mathcal P_k(J)}d_{\mathrm H}(w,y).
\]
We shall use the two maxima
\begin{equation}\label{union:palette-radii}
 \beta_k(J)=\max_{\mathbf p\in\mathcal M_k(J)}
                      \rho_{J,k}(\mathbf p),
 \qquad
 F_k(J)=\max_{\mathbf a\in\mathcal A_k(J)}
                      \rho_{J,k}(\mathbf a).
\end{equation}
The first is weak villainy with the fixed palette $[k]$; in particular,
$\beta_{\chi(J)}(J)=B_{\mathrm w}(J)$.
The second is the $k$-fixing number studied by
Garnero et al.~\cite{garnero2018}, in which the source coloring is
unrestricted.  Put
\[
 \varepsilon_k(J)=F_k(J)-\beta_k(J)\ge0.
\]
This difference measures how much is lost by requiring the source
profile to be feasible.
\end{definition}

\Needspace{6\baselineskip}
\begin{definition}[Transfer surplus]\label{def:transfer-surplus}
For an integral vector $\boldsymbol\delta$ whose coordinates sum to
zero, define the \emph{transfer surplus}
\begin{equation}\label{union:surplus-definition}
 \Delta_{J,k}(\boldsymbol\delta)=
 \max_{\substack{\mathbf p\in\mathcal M_k(J)\\
                 \mathbf p+\boldsymbol\delta\ge\mathbf0}}
       \rho_{J,k}(\mathbf p+\boldsymbol\delta)-\beta_k(J).
\end{equation}
An empty maximum has value $-\infty$.  The vector
$\boldsymbol\delta$ records a change in the local color counts relative
to a proper coloring.  Its coordinates can have both signs, although
their sum is zero because the order of the summand is unchanged.
For a union, this change must be offset by
$-\boldsymbol\delta$ on the other summand.
\end{definition}

\begin{theorem}\label{union:transfer-theorem}
Let $k\ge\max\{\chi(G),\chi(H)\}$, and put
\[
 T_k(G,H)=\max_{\boldsymbol\delta}
  \bigl(\Delta_{G,k}(\boldsymbol\delta)
       +\Delta_{H,k}(-\boldsymbol\delta)\bigr).
\]
Then
\begin{equation}\label{union:transfer-identity}
 \beta_k(G\sqcup H)=\beta_k(G)+\beta_k(H)+T_k(G,H),
\end{equation}
where
\begin{equation}\label{union:transfer-bounds}
 0\le T_k(G,H)\le\varepsilon_k(G)+\varepsilon_k(H).
\end{equation}
\end{theorem}

\begin{proof}
A proper coloring of $G\sqcup H$ is exactly a pair of local proper
colorings, and hence
\[
 \mathcal M_k(G\sqcup H)=\mathcal M_k(G)+\mathcal M_k(H).
\]
For fixed local source words, the two weak repairs are independent and
their costs add.  Maximizing over sources of prescribed local profiles
therefore gives
\begin{equation}\label{union:weak-convolution}
 \beta_k(G\sqcup H)=
 \max_{\substack{\mathbf a\in\mathcal A_k(G),\quad
                  \mathbf c\in\mathcal A_k(H)\\
         \mathbf a+\mathbf c\in\mathcal M_k(G)+\mathcal M_k(H)}}
  \bigl(\rho_{G,k}(\mathbf a)+\rho_{H,k}(\mathbf c)\bigr).
\end{equation}
For every admitted pair $\mathbf a,\mathbf c$, choose proper profiles
$\mathbf p,\mathbf q$ with
$\mathbf p+\mathbf q=\mathbf a+\mathbf c$.  Then
\[
 \boldsymbol\delta=\mathbf a-\mathbf p
                  =\mathbf q-\mathbf c
\]
has coordinate sum zero.  Conversely, proper profiles
$\mathbf p,\mathbf q$ and a vector $\boldsymbol\delta$ for which
$\mathbf p+\boldsymbol\delta$ and
$\mathbf q-\boldsymbol\delta$ are nonnegative determine an admitted
source pair.  For fixed $\boldsymbol\delta$, the choices of
$\mathbf p$ and $\mathbf q$ are independent.  Substitution into
\eqref{union:weak-convolution} and separate maximization over these
two profiles prove~\eqref{union:transfer-identity}.

The zero transfer satisfies $\Delta_{J,k}(\mathbf0)=0$, so
$T_k(G,H)\ge0$.  Also
$\rho_{J,k}(\mathbf p+\boldsymbol\delta)\le F_k(J)$,
which bounds each surplus by $\varepsilon_k(J)$ and proves the upper
bound in~\eqref{union:transfer-bounds}.
\end{proof}

The theorem identifies two distinct effects in the original
superadditivity question.  A common palette changes the individual
weak-villainy values to $\beta_k(G)$ and $\beta_k(H)$.  Relative to
these common-palette values, transferring color counts can only
increase the worst repair cost, and its entire contribution is
$T_k(G,H)$.

\begin{corollary}\label{union:weak-equality-general}
If $\chi(G)=\chi(H)=k$, then
\[
 B_{\mathrm w}(G\sqcup H)=B_{\mathrm w}(G)+B_{\mathrm w}(H)
\]
if and only if
\begin{equation}\label{union:weak-equality-transfers}
 \Delta_{G,k}(\boldsymbol\delta)
 +\Delta_{H,k}(-\boldsymbol\delta)\le0
 \qquad\text{for every }\boldsymbol\delta.
\end{equation}
In particular, equality holds if
$F_k(G)=B_{\mathrm w}(G)$ and $F_k(H)=B_{\mathrm w}(H)$.

For arbitrary $G,H$, set $k=\max\{\chi(G),\chi(H)\}$ and
\[
 \Gamma_k(G,H)=B_{\mathrm w}(G)+B_{\mathrm w}(H)
                   -\beta_k(G)-\beta_k(H).
\]
Then
\begin{equation}\label{union:weak-defect-general}
 B_{\mathrm w}(G\sqcup H)-B_{\mathrm w}(G)-B_{\mathrm w}(H)
 =T_k(G,H)-\Gamma_k(G,H).
\end{equation}
Consequently superadditivity holds if and only if
$T_k(G,H)\ge\Gamma_k(G,H)$, and equality holds if and only if
$T_k(G,H)=\Gamma_k(G,H)$.
\end{corollary}

\begin{proof}
At equal chromatic numbers the palette correction is zero, and the
maximum defining $T_k$ already contains the zero transfer.  Thus
$T_k=0$ is equivalent to~\eqref{union:weak-equality-transfers}.
The sufficient condition follows from~\eqref{union:transfer-bounds}.
For unequal chromatic numbers, subtract the two intrinsic weak-villainy
values from~\eqref{union:transfer-identity}.
\end{proof}

These criteria use local profile radii and remain valid for every
graph pair.  The correction $\Gamma_k(G,H)$ is kept signed; the
argument does not require an assumption about how the feasible-source
maximum changes when additional colors are made available.

\subsection{Repeated copies and the fixing number}
\label{union:repeated-copies}

The unrestricted maximum $F_k(G)$ may be attained by a word whose
profile is not feasible.  Taking several copies of the graph permits
such a word to be balanced across the palette.  This gives an exact
relation between weak villainy and the fixing number.

\begin{theorem}\label{union:repetition-theorem}
Fix an integer $k\ge\max\{1,\chi(G)\}$, and write $m=qk+r$, where
$q\ge0$ and $0\le r<k$.  Then
\begin{equation}\label{union:repetition-error}
 0\le mF_k(G)-\beta_k(mG)
 \le r\bigl(F_k(G)-\beta_k(G)\bigr).
\end{equation}
In particular,
\begin{equation}\label{union:repetition-exact}
 \beta_k(mG)=mF_k(G)\qquad\text{whenever }k\mid m,
\end{equation}
and
\begin{equation}\label{union:repetition-limit}
 \lim_{m\to\infty}\frac{\beta_k(mG)}m=F_k(G).
\end{equation}
\end{theorem}

\begin{proof}
Every source on $mG$ can be weakly repaired one copy at a time, each
at cost at most $F_k(G)$.  This proves the upper bound
$\beta_k(mG)\le mF_k(G)$.

Let $w$ be a word on $G$ at distance $F_k(G)$ from
$\mathcal P_k(G)$, and let $\pi$ be a cyclic permutation of all $k$
color labels.  On $k$ copies of $G$, use the sources
\[
 w,\ \pi\circ w,\ \ldots,\ \pi^{k-1}\circ w.
\]
If $G$ has $n$ vertices, their total profile is
$(n,\ldots,n)$: the $k$ copies of every vertex receive every label
once.  This profile is feasible.  Indeed, starting with any proper
$[k]$-coloring $y$ of $G$, its $k$ cyclic relabelings give a proper coloring
of $kG$ with the same total profile.
Relabeling preserves distances to the proper-coloring set, so each
of the $k$ sources still requires $F_k(G)$ changes in a weak repair.
The repairs on distinct copies are independent.  Thus this admitted
source on $kG$ requires exactly $kF_k(G)$ changes.

Repeating the construction gives $qk$ copies with cost $qkF_k(G)$.
On each of the remaining $r$ copies, use a feasible source attaining
$\beta_k(G)$.  The resulting total profile is feasible, since
every block and each remaining copy has a feasible profile.
Its weak repair cost is
\[
 qkF_k(G)+r\beta_k(G).
\]
Combining this lower bound with the upper bound proves
\eqref{union:repetition-error}.  The exact formula when $r=0$ follows
immediately, and the error in~\eqref{union:repetition-error} is bounded
independently of $m$, proving~\eqref{union:repetition-limit}.
\end{proof}

When $k=\chi(G)$, the common-palette quantities in this theorem
are the intrinsic weak-villainy values.

The error in \cref{union:repetition-error} eventually repeats. The
following result gives an explicit stabilization bound and identifies
its least period from the repetitions attaining the fixing-number bound.

\begin{theorem}[Effective weak repetition]\label{union:repetition-stabilization}
\label{union:repetition-periodic}
Let $G$ have $n$ vertices and let $k\ge\max\{2,\chi(G)\}$.
Put
\[
 F=F_k(G),\qquad \varepsilon=F-\beta_k(G),\qquad
 M=\bigl(2(k-1)n+1\bigr)^{k-1},\qquad N=(k-1)M.
\]
Define
\[
 h=\gcd\{\ell\ge1:\beta_k(\ell G)=\ell F\}.
\]
Then $h\mid k$, and there are integers $c_0,\ldots,c_{h-1}$ such that
\[
 c_0=0,\qquad 1\le c_r\le r\varepsilon\quad(1\le r<h),
 \qquad
 \beta_k(mG)=mF-c_{m\bmod h}\quad(m\ge N).
\]
The number $h$ is the least eventual period of the defect sequence
$d_m=mF-\beta_k(mG)$: it is the least positive integer for which
$d_{m+h}=d_m$ for all sufficiently large $m$. In particular, for $m\ge N$,
\[
 \beta_k(mG)=mF\quad\Longleftrightarrow\quad h\mid m.
\]
Moreover,
\[
 h=\gcd\{1\le\ell\le M:\beta_k(\ell G)=\ell F\}.
\]
The period and coefficients are computable from the feasible profiles
and the common-palette radii of $G$.
\end{theorem}

\begin{proof}
Write $d=k-1$. We use the Steinitz lemma in the following standard
form: vectors in $\mathbb R^d$ of $\ell_\infty$ norm at most $n$
and with sum zero can be ordered so that every partial sum has
$\ell_\infty$ norm at most $dn$; see \cite{eisenbrandweismantel2018}.
Consequently a nonempty zero-sum multiset of integer vectors of norm
at most $n$ that has no proper nonempty zero-sum submultiset contains
at most $(2dn+1)^d=M$ vectors. Indeed, for a multiset of length $L$,
the partial sums $\mathbf s_0=\mathbf0,\mathbf s_1,\ldots,\mathbf s_{L-1}$
in the Steinitz ordering
are distinct lattice points in $[-dn,dn]^d$.
An equality between two of these sums would produce
a proper nonempty zero-sum submultiset.

To each pair
$(\mathbf a,\mathbf p)\in\mathcal A_k(G)\times\mathcal M_k(G)$,
associate the vector
\[
 \boldsymbol\delta=(a_1-p_1,\ldots,a_d-p_d)
\]
and nonnegative integer cost $F-\rho_{G,k}(\mathbf a)$.
The omitted coordinate of $\mathbf a-\mathbf p$ is determined by
the first $d$, because its coordinate sum is zero.
Thus a multiset of $m$ local pairs has vector sum zero precisely
when its source profiles sum to its proper profiles. Such a choice
is exactly a feasible global source profile on $mG$.
For fixed local source profiles, the maximum weak repair cost is
the sum of the local radii, since sources can be chosen independently
and weak repairs on different copies are independent. It follows that
\[
 d_m:=mF-\beta_k(mG)
\]
is the minimum total cost of a zero-sum multiset of $m$ local pairs.
The length of a multiset is its number of pairs, counted with
multiplicity, and its cost is the sum of their costs.
We also allow the empty multiset and put $d_0=0$.

Call a nonempty zero-sum multiset a block if it has no proper
nonempty zero-sum submultiset. Every zero-sum multiset is a disjoint
union of blocks, by successively taking a minimal nonempty zero-sum
submultiset. Each block has length at most $M$, by the preceding
Steinitz consequence. There are only finitely many blocks because
the set of local pairs is finite.

Consider the directed graph with vertex set $\mathbb Z/k\mathbb Z$.
For every block of length $\ell$ and cost $b$, place an edge from
each residue $u$ to $u+\ell$, with weight $b$.
A path starting at zero represents a concatenation of blocks, so
the shortest-path distance $a_r$ from zero to $r$ is
\[
 a_r=\min_{m\ge0,\ m\equiv r\pmod k}d_m.
\]
Every residue is reachable: a local pair with $\mathbf a=\mathbf p$
is itself a block of length one. Edge weights are nonnegative, so
a shortest path can be chosen without repeated vertices. It has
at most $k-1$ edges and therefore represents a witness of length
$m_r\le(k-1)M=N$ and cost $a_r$.

The cyclic relabeling construction in
\cref{union:repetition-theorem} gives a zero-cost multiset of length
$k$. For any $m\ge N$ with $m\equiv r\pmod k$, append
$(m-m_r)/k$ copies of that multiset to the witness above. This gives
$d_m\le a_r$, while the definition of $a_r$ gives the reverse
inequality. Hence
\begin{equation}\label{union:explicit-periodic-defect}
 d_m=a_{m\bmod k}\qquad(m\ge N).
\end{equation}
The original error bound also gives $a_0=0$ and
$a_r\le r\varepsilon$ for $0\le r<k$.

To identify the least period, let $H$ be the subgroup of
$\mathbb Z/k\mathbb Z$ generated by the lengths of zero-cost
blocks. Nonnegative edge weights imply that $a_r=0$ exactly when
$r$ is reachable from zero using only zero-cost edges. In a finite
group, the semigroup generated by given residues is their subgroup;
thus $a_r=0$ exactly when $r\in H$.
Appending a zero-cost path shows that $a_{r+u}\le a_r$ for $u\in H$.
Applying the same argument with $-u\in H$ proves equality.
Consequently $(a_r)$ is constant on the cosets of $H$.

The zero-cost multiset of length $k$ decomposes into zero-cost blocks,
because all costs are nonnegative. Hence the gcd of the lengths of
zero-cost blocks divides $k$. Every zero-cost multiset also decomposes
into zero-cost blocks. Its possible lengths are therefore precisely
the additive semigroup generated by those block lengths, that is,
their nonnegative integer combinations, and their
gcd is the number $h$ in the statement. Thus $H=h\mathbb Z/k\mathbb Z$.
The sequence $(a_r)$ has period $h$. Any period $t$ must send its
zero entry $a_0$ to another zero entry, so $t\bmod k\in H$; hence $h$ is
its least positive period. By \eqref{union:explicit-periodic-defect},
it is also the least eventual period of $(d_m)$.

Set $c_r=a_r$ for $0\le r<h$. These coefficients satisfy the stated
bounds, and $c_r>0$ when $r\ne0$ because such a residue does not
belong to $H$. Every zero-cost block has length at most $M$ and is
itself a zero-cost multiset. Conversely, every zero-cost multiset
of length at most $M$ decomposes into zero-cost blocks. The two
collections of lengths have the same gcd, proving the finite
formula for $h$.

Finally, all local pairs, their costs, and all blocks of length at
most $M$ can be enumerated from the displayed local data. The
shortest-path construction then computes the coefficients, and
the zero-cost block lengths compute $h$. This is a finite
procedure; no polynomial bound on the enumeration is asserted.
\end{proof}

When $k=1$, the graph is edgeless and $F_1(G)=\beta_1(mG)=0$.


\begin{corollary}\label{union:all-repetition-equality}
For a graph $G$ with $\chi(G)=k$, the following are equivalent:
\begin{enumerate}
\item $B_{\mathrm w}(G)=F_k(G)$;
\item $B_{\mathrm w}(mG)=mB_{\mathrm w}(G)$ for every positive integer $m$;
\item $B_{\mathrm w}(kG)=kB_{\mathrm w}(G)$.
\end{enumerate}
\end{corollary}

\begin{proof}
If the first equality holds, the error bound in
\cref{union:repetition-theorem} is zero for every $m$.
The second statement includes the third.  Finally,
\eqref{union:repetition-exact} with $m=k$ shows that the third implies
the first.
\end{proof}

Thus additivity over every repetition is equivalent to a property of
a single graph: an unrestricted worst source can be chosen with a
feasible profile.

\subsection{Count-preserving repairs and response functions}

For ordinary villainy, the target profiles on the summands are coupled
by their prescribed sum. We retain the common palette of $k$ colors
and the profile sets $\mathcal A_k(J),\mathcal M_k(J)$ introduced
in \cref{def:common-palette}. Local source profiles range
over $\mathcal A_k(J)$, even when they are not proper profiles.

\begin{definition}[Repair responses]\label{def:responses}
The \emph{response function} of a word $w:V(J)\to[k]$ of profile
$\mathbf a$ is
\[
 f_w(\mathbf p)=
 \min_{\substack{y\in\mathcal P_k(J)\\\mathbf m(y)=\mathbf p}}
       d_{\mathrm H}(w,y)
 \qquad(\mathbf p\in\mathcal M_k(J)).
\]
Here all colorings and words use the common palette; a proper coloring
of a summand need not use every color. Extend $f_w$ by $+\infty$
outside $\mathcal M_k(J)$. Write
\[
 \mathcal R_k(J,\mathbf a)
 =\operatorname{Max}\{f_w:\mathbf m(w)=\mathbf a\},
\]
where $\operatorname{Max}$ retains the coordinatewise maximal
functions and removes duplicates. Thus a response is discarded if
another realizable response is at least as large at every target
profile and larger at some profile. The resulting family is an
\emph{antichain}: no two distinct remaining responses are comparable
coordinatewise. It records the possible tradeoffs between target
profiles without retaining the source words themselves.

For two response functions define their min-plus convolution by
\[
 (f\star g)(\mathbf p)
 =\min_{\mathbf u+\mathbf v=\mathbf p}
       \bigl(f(\mathbf u)+g(\mathbf v)\bigr).
\]
The profiles where its value is finite are exactly
$\mathcal M_k(G)+\mathcal M_k(H)$ when $f,g$ are responses on
$G,H$, respectively. For sets of responses, put
$\mathcal U\star\mathcal V=\{f\star g:f\in\mathcal U,
g\in\mathcal V\}$.
\end{definition}

\Needspace{6\baselineskip}
\begin{theorem}\label{union:ordinary-composition}
For every $k\ge\max\{\chi(G),\chi(H)\}$ and every source profile
$\mathbf a$ on $G\sqcup H$,
\begin{equation}\label{union:frontier-composition}
\mathcal R_k(G\sqcup H,\mathbf a)
=\operatorname{Max}
\bigcup_{\boldsymbol\alpha+\boldsymbol\beta=\mathbf a}
\bigl(\mathcal R_k(G,\boldsymbol\alpha)
      \star\mathcal R_k(H,\boldsymbol\beta)\bigr),
\end{equation}
where the union ranges over
$\boldsymbol\alpha\in\mathcal A_k(G)$ and
$\boldsymbol\beta\in\mathcal A_k(H)$.
This composition is associative and commutative.
\end{theorem}

\begin{proof}
A source word on the union restricts to words $w,z$ on the two
summands. At a fixed target profile $\mathbf p$, any proper repair
has local target profiles $\mathbf u,\mathbf v$ with
$\mathbf u+\mathbf v=\mathbf p$. Once these profiles are fixed,
the two repairs are independent and their distances add. The union
response is therefore exactly $f_w\star f_z$. Conversely, any pair
of local response functions can be realized simultaneously by choosing
its two source words independently.

This proves~\eqref{union:frontier-composition} before dominated
responses are removed. Convolution is monotone in each argument:
if $f\le f'$ and $g\le g'$, then $f\star g\le f'\star g'$.
Replacing a discarded local response by a maximal one can consequently
only increase the resulting response. Pruning the local lists thus
leaves the union's maximal responses unchanged. Associativity and
commutativity follow from those of min-plus convolution, and the same
monotonicity permits pruning at every intermediate composition.
\end{proof}

This theorem is a reduction to finite local data; it does not assert
that the response antichains are always small. When $k=\chi(J)$,
ordinary villainy is recovered from their diagonal entries:
\begin{equation}\label{union:response-diagonal}
 B(J)=\max_{\mathbf a\in\mathcal M_k(J)}
       \max_{f\in\mathcal R_k(J,\mathbf a)}f(\mathbf a).
\end{equation}

The antichain is a singleton whenever one source attains every
coordinatewise maximum simultaneously. The common extremizer in
\cite[Theorem~3.15]{jiang2026exact} gives
\[
\begin{aligned}
 \mathcal R_3(C_{2r+1},\mathbf a)&=\{f_{\mathbf a}\},\\
 f_{\mathbf a}(\mathbf p)&=2r+1-
 \sum_{X\in[3]}\min\{p_X,\lceil a_X/2\rceil\}
 \qquad(\mathbf p\in\mathcal M_3(C_{2r+1}))
\end{aligned}
\]
for every source profile $\mathbf a$, including infeasible ones:
the same cyclic word, with all occurrences of each color consecutive,
attains the displayed value at every target profile.
Thus these directed radii supply the complete local response
data needed for composition.

In general, however, it is essential to retain the antichain, rather than merely the largest
value in each coordinate. For example, on $P_3=K_{1,2}$ with two
colors, order the target profiles by their $A$-counts $1,2$. Sources
with one $A$ have responses $(0,3)$ and $(2,1)$, according as $A$
occupies the center or a leaf. Sources with two $A$'s have responses
$(3,0)$ and $(1,2)$. For two copies and the source split $1+2$,
the greatest convolution value at total $A$-count three is two.
Taking coordinatewise maxima first instead gives
\[
 \min\{2+2,3+3\}=4.
\]
Indeed, the actual radius of profile $(3,3)$ on $2P_3$ is two:
its two proper targets are complementary, their distances from a
same-profile source are even and sum to six, and a source constant
on each copy attains two. Thus the outer maximum over disturbances
cannot in general be interchanged with the minimum over target profiles.

The response functions also give an exact criterion for equality
under disjoint union. For an integer $j$, put
\[
 \mathcal L_f(j)=\{\mathbf p:f(\mathbf p)\le j\},
 \qquad
 \mathcal K_L(f,g)=
 \bigcup_{j=0}^{L}
 \bigl(\mathcal L_f(j)+\mathcal L_g(L-j)\bigr),
\]
with $\mathcal L_f(j)=\varnothing$ for $j<0$ and
$\mathcal K_{-1}(f,g)=\varnothing$.

\begin{proposition}\label{union:ordinary-equality}
Set $k=\chi(G\sqcup H)$ and let $L\ge0$ be an integer.
Then $B(G\sqcup H)\le L$ if and only if
\begin{equation}\label{union:budget-cover}
 \boldsymbol\alpha+\boldsymbol\beta\in\mathcal K_L(f,g)
\end{equation}
for every
$f\in\mathcal R_k(G,\boldsymbol\alpha)$ and
$g\in\mathcal R_k(H,\boldsymbol\beta)$ whose summed source profile
belongs to $\mathcal M_k(G)+\mathcal M_k(H)$.
In particular, for $L=B(G)+B(H)$, equality
$B(G\sqcup H)=B(G)+B(H)$ holds if and only if this condition holds
for every such pair at budget $L$, and fails for at least one such
pair at budget $L-1$.
\end{proposition}

\begin{proof}
For a fixed pair of responses,
$(f\star g)(\mathbf m)\le L$ means that
$\mathbf m=\mathbf p+\mathbf q$ for profiles satisfying
$f(\mathbf p)+g(\mathbf q)\le L$. Since costs are integers, this
is equivalent to $\mathbf m\in\mathcal K_L(f,g)$.
\cref{union:ordinary-composition} and
\eqref{union:response-diagonal} now give the asserted upper-bound
criterion. Failure at $L-1$, together with the bound at $L$, is
equivalent to equality. The argument does not require equal chromatic
numbers for the summands.
\end{proof}

This is an exact compositional criterion in local profile data, rather
than an elementary structural classification of the graph pairs.
It distinguishes two effects: target multiplicities may transfer
between summands during repair, and a globally feasible disturbance
may have locally infeasible profiles.

\subsection{The growth of ordinary villainy under repetition}
\label{union:ordinary-repetition-section}

The response composition law also determines the asymptotic behavior
of ordinary villainy. Fix a graph $G$ with $\chi(G)=k$, and let
$\mathcal W$ and $\mathcal P$ be, respectively, all words and all
proper colorings of $G$ on $[k]$. Both sets are finite. A vector
$\mathbf x=(x_w)_{w\in\mathcal W}$ describes proportions of source
words on many copies of $G$. Its average profile must belong to the
convex hull of the proper profiles. Accordingly, put
\[
 \mathcal X_G=\left\{\mathbf x\ge\mathbf0:
   \sum_wx_w=1,\quad
   \sum_wx_w\mathbf m(w)\in\operatorname{conv}\mathcal M_k(G)
                 \right\}.
\]
For $\mathbf x\in\mathcal X_G$, define the fractional repair cost
\begin{equation}\label{union:fractional-response}
\begin{aligned}
 \Phi_G(\mathbf x)=\min_{\mathbf z\ge\mathbf0}\quad&
   \sum_{w\in\mathcal W,\ y\in\mathcal P}d_{\mathrm H}(w,y)z_{wy},\\
 \text{subject to}\quad&\sum_yz_{wy}=x_w\quad(w\in\mathcal W),\\
 &\sum_{w,y}z_{wy}\bigl(\mathbf m(y)-\mathbf m(w)\bigr)=\mathbf0.
\end{aligned}
\end{equation}
The variable $z_{wy}$ assigns a proportion of copies with source $w$
to the proper target $y$. Thus the second constraint preserves the
total color counts. The program is feasible precisely for
$\mathbf x\in\mathcal X_G$. Set
\begin{equation}\label{union:ordinary-slope}
 \Lambda(G)=\max_{\mathbf x\in\mathcal X_G}\Phi_G(\mathbf x).
\end{equation}

\begin{theorem}[Ordinary repetition]\label{union:ordinary-repetition}
For every fixed graph $G$ with $k=\chi(G)$, the number $\Lambda(G)$
is rational and effectively computable. The maximum
in~\eqref{union:ordinary-slope} is attained by a distribution supported
on at most $k$ source words, meaning that at most $k$ coordinates
$x_w$ are positive. Moreover,
\begin{equation}\label{union:ordinary-linear-growth}
B(mG)=m\Lambda(G)+O_G(1),\qquad \Lambda(G)\ge F_k(G).
\end{equation}
The error term is bounded by a constant depending only on $G$.
There are effectively computable integers $h\ge1$, $a\ge0$, and
$M\ge1$ such that $a=h\Lambda(G)$ and
\begin{equation}\label{union:ordinary-eventual-recurrence}
 B((m+h)G)=B(mG)+a\qquad(m\ge M).
\end{equation}
\end{theorem}

\begin{proof}
Write the constraints in~\eqref{union:fractional-response} as
$A\mathbf z=(\mathbf x,\mathbf0)$, where $A$ is a fixed integer
matrix, and write $\mathbf c$ for the nonnegative integer cost vector.
When $\mathbf x\in\mathbb Z_{\ge0}^{\mathcal W}$ counts actual
source copies, its sum is their number. The source is feasible if
and only if the same system has a nonnegative integer solution.
In that case the minimum over integer solutions, denoted
$C_G(\mathbf x)$, is its ordinary repair cost. This is the repeated
form of \cref{union:ordinary-composition}. For unnormalized vectors,
we define $\Phi_G$ by the same program; scaling all source proportions
by a nonnegative factor scales its optimum by that factor.

For this fixed matrix and cost vector, the integer-programming
proximity bound of Cook, Gerards, Schrijver, and Tardos
\cite[Corollary~2]{cook1986} gives a constant $D$, independent of
$\mathbf x$, such that
\begin{equation}\label{union:fixed-rounding}
 0\le C_G(\mathbf x)-\Phi_G(\mathbf x)\le D
\end{equation}
whenever an integer solution exists. To apply that bound, write the
equalities as pairs of inequalities and include nonnegativity.
The resulting integral constraint matrix and objective are fixed,
and the row sums bound every feasible variable.

The function $\Phi_G$ is convex and piecewise linear by linear
programming duality. Its maximum on the rational polytope
$\mathcal X_G$ is attained at a vertex. Such a vertex has at most
$k$ positive coordinates: more than $k$ corresponding profile vectors
would be linearly dependent, and a sufficiently small perturbation
in that dependence would preserve both the average profile and the
sum of the coordinates. Enumerating the vertices and solving the
corresponding rational linear programs computes a rational value
$\Lambda(G)$. Any feasible integral source count vector with sum
$m$ has $\mathbf x/m\in\mathcal X_G$. Thus
\eqref{union:fixed-rounding} gives
$B(mG)\le m\Lambda(G)+D$.

For the reverse bound, choose a rational maximizer
$\mathbf x^*\in\mathcal X_G$. Express its average profile as a
rational convex combination of proper profiles, and choose a common
denominator $d$ for these coefficients and $\mathbf x^*$.
Then $d\mathbf x^*$ is the source count vector of a feasible source
on $dG$: the resulting integer numbers of source copies and proper
target copies have the same total profile and may be paired
arbitrarily. Fix a proper word $y_0$. For $m=qd+r$, $0\le r<d$, the
integer vector
$\mathbf x^{(m)}=qd\mathbf x^*+r\mathbf e_{y_0}$ is feasible.
Take an optimal dual solution at $\mathbf x^*$, and let
$\mathbf u$ denote its coordinates corresponding to the row-sum
constraints. Dual feasibility holds for every source count vector,
so
\[
 B(mG)\ge C_G(\mathbf x^{(m)})
 \ge\Phi_G(\mathbf x^{(m)})
 \ge qd\Lambda(G)+ru_{y_0}
 =m\Lambda(G)-r\bigl(\Lambda(G)-u_{y_0}\bigr).
\]
The final error is bounded independently of $m$, proving the first
assertion of~\eqref{union:ordinary-linear-growth}. The second follows
from $B(mG)\ge B_{\mathrm w}(mG)$ and
\cref{union:repetition-theorem}.

Finally, the set of pairs $(m,b)$ with $b=B(mG)$ is definable in
Presburger arithmetic, the language of integer linear inequalities
and congruences combined by Boolean operations and quantifiers.
Indeed, feasible sources,
integer repairs, and their costs are described by the fixed linear
system above. Let $P(m,t)$ assert that some feasible integer source
count vector of sum $m$ has every integer repair of cost at least $t$.
This is a Presburger formula, and $B(mG)=b$ is expressed by
$P(m,b)\wedge\neg P(m,b+1)$. Write $\Lambda(G)=a_0/h_0$ with
integers $a_0\ge0$ and $h_0\ge1$.
By~\eqref{union:ordinary-linear-growth}, the integer
$h_0B(mG)-a_0m$ belongs to a fixed finite set.
For each of its possible values, the set of corresponding $m$ is a
one-variable Presburger set, hence effectively eventually periodic
by effective quantifier elimination~\cite{ginsburg1966,haase2024},
which replaces a quantified formula by an equivalent formula without
quantifiers. After elimination, each
linear inequality is eventually constant and each congruence is
periodic. Taking a common period $h$ divisible by $h_0$ and a common
threshold proves~\eqref{union:ordinary-eventual-recurrence}.
The proximity bound and the bound obtained from the dual solution
are effective, so
the finite set of possible errors can be bounded before applying
quantifier elimination.
\end{proof}

Comparing this slope with the fixing-number slope in
\cref{union:repetition-theorem} determines whether the cost of preserving
color counts accumulates under repetition.

\begin{corollary}\label{union:repetition-gap}
For a fixed graph $G$ with $k=\chi(G)$,
\[
 B(mG)-B_{\mathrm w}(mG)
   =m\bigl(\Lambda(G)-F_k(G)\bigr)+O_G(1).
\]
The difference is bounded if and only if $\Lambda(G)=F_k(G)$.
In that case it is eventually periodic; otherwise it tends to
infinity linearly. The boundedness condition is decidable from the
finite linear programs defining $\Lambda(G)$ and the finite
definition of $F_k(G)$.
\end{corollary}
\begin{proof}
Subtract \cref{union:repetition-theorem} from
\eqref{union:ordinary-linear-growth}. When the two slopes agree,
\cref{union:repetition-periodic,union:ordinary-repetition} give a
common period for the difference.
\end{proof}

This gives a finite optimization criterion for bounded difference.
A structural characterization of the graphs satisfying
$\Lambda(G)=F_{\chi(G)}(G)$ remains a separate question.


\section{Fixed-parameter recognition}\label{rec:section}

The composition laws in \cref{union:composition-section} retain exactly
the local information needed for repair, but those data can be large.
We now address computation on arbitrary graphs. Small weak villainy
forces a small vertex cover in the graph or its complement, so the
relevant colorings can be described by boundedly many vertex types.

The classifications of graphs with $B(G)=0$ in
\cite{clark2006} and with $B(G)=2$ in \cite{jahanbekam2021}
give polynomial-time recognition algorithms for those two classes.
We prove the stronger conclusion that recognition is fixed-parameter
tractable in the threshold, for both repair rules.

\begin{theorem}[Exact computation and recognition]\label{rec:main}
\label{rec:exact-parameter}
There are computable functions $f,h$ and an absolute constant $c$
such that both $B(G)$ and $B_{\mathrm w}(G)$ can be computed exactly
on an $n$-vertex graph in time
\[
 h(B_{\mathrm w}(G))(n+1)^c,
\]
without either value being supplied. Moreover, given an integer
$t\geq0$, both
\[
\min\{B(G),t+1\}\qquad\text{and}\qquad
\min\{B_{\mathrm w}(G),t+1\}
\]
can be computed in time $f(t)(n+1)^c$.
Consequently, equality and threshold recognition for either parameter
are fixed-parameter tractable in $t$, and belong to $\mathrm P$
for every fixed $t$.
\end{theorem}

The exponent $c$ is independent of the threshold. The proof first forces
a small vertex cover in the graph or its complement, and then uses
different descriptions of repair in the two cases. The small-cover
case uses effective quantifier elimination; no practical bound on
$f$ is asserted. Throughout this section,
$\nu(H)$ denotes the matching number of a graph $H$.

\subsection{A structural consequence of a matching bound}

For an optimal coloring profile $\mathbf m=(m_1,\ldots,m_k)$, where
$k=\chi(G)$, put
\[
p(\mathbf m)=\sum_{i=1}^k\left\lfloor\frac{m_i}{2}\right\rfloor.
\]
This is the number of disjoint pairs of equal-colored tokens available
in a word of profile $\mathbf m$. The next lemma extends the matching
argument of Jahanbekam and Lin~\cite[Lemma~8]{jahanbekam2021}.

\Needspace{6\baselineskip}
\begin{lemma}\label{rec:matching}
For every optimal coloring profile $\mathbf m$ of $G$,
\[
B(G)\geq B_{\mathrm w}(G)\geq\min\{p(\mathbf m),\nu(G)\}.
\]
\end{lemma}
\begin{proof}
Let $q=\min\{p(\mathbf m),\nu(G)\}$. Choose $q$ edges of a matching
and $q$ disjoint pairs of equal-colored tokens. Place each pair at the
ends of one chosen edge and distribute the remaining tokens arbitrarily.
The resulting word has profile $\mathbf m$. Every proper coloring
differs from it at an endpoint of each of the $q$ selected edges.
The edges are disjoint, so these discrepancies occur at distinct
vertices. This argument does not require the repaired coloring to
have profile $\mathbf m$, and therefore proves the bound for $B_{\mathrm w}(G)$.
\end{proof}

The matching bound forces a graph of small weak villainy to have
either a small vertex cover or very many colors. In the latter case,
the complement has a small vertex cover instead.

\begin{corollary}\label{rec:dichotomy}
If $B_{\mathrm w}(G)\leq t$, then
\[
\nu(G)\leq t\quad\text{or}\quad\chi(G)\geq n-2t.
\]
In the first case $G$ has a vertex cover of size at most $2t$.
In the second case $\nu(\overline G)\leq 2t$, and
$\overline G$ has a vertex cover of size at most $4t$.
The same conclusions hold if $B(G)\leq t$.
\end{corollary}
\begin{proof}
Fix an optimal profile $\mathbf m$. Since every $m_i$ is positive,
\[
n-\chi(G)=\sum_i(m_i-1)
\leq 2\sum_i\left\lfloor\frac{m_i}{2}\right\rfloor
=2p(\mathbf m).
\]
If $\nu(G)>t$, \cref{rec:matching} forces $p(\mathbf m)\leq t$,
and hence $\chi(G)\geq n-2t$.

The endpoints of a maximum matching form a vertex cover: otherwise
an edge with both endpoints uncovered could be added to the matching.
This gives the first cover bound. Moreover, a matching of size $q$ in
$\overline G$ supplies $q$ disjoint nonadjacent pairs in $G$. Giving
each pair one color and every remaining vertex its own color gives
$\chi(G)\leq n-q$. Thus $\nu(\overline G)\leq n-\chi(G)\leq 2t$
in the second case, and the same endpoint argument gives its cover
bound. Finally, $B_{\mathrm w}(G)\leq B(G)$.
\end{proof}

\subsection{Descriptions of colorings by types}\label{rec:types}

Let $U\subseteq V(G)$, with $|U|=s$, and suppose that
$V(G)\setminus U$ is either an independent set or a clique.
For a vertex $v\notin U$, define its type to be
$\tau(v)=N_G(v)\cap U$. Write $V_\tau$ for the vertices of type
$\tau$ and $n_\tau=|V_\tau|$, keeping only nonempty types.
There are at most $2^s$ types. Any permutation fixing $U$ pointwise
and preserving each $V_\tau$ is an automorphism of $G$; let $\Gamma$
be the group of these permutations.

For a finite label set $L$, describe a map $x:V(G)\to L$ by its
restriction to $U$, called its \emph{core map}, and by the vectors
\[
\mathbf x_\tau=(x_{\tau i})_{i\in L},\qquad
x_{\tau i}=|V_\tau\cap x^{-1}(i)|.
\]
Two maps have the same description if and only if an element of
$\Gamma$ takes one to the other. The next observation is the reason
that descriptions can replace individual words in a repair problem.

\begin{lemma}\label{rec:distance}
Let $\mathcal A$ and $\mathcal C$ be descriptions of maps $x,y:V(G)\to L$.
Define
\[
\delta(\mathcal A,\mathcal C)
=|\{u\in U:x(u)\neq y(u)\}|
+\sum_\tau\left(n_\tau-\sum_{i\in L}
\min\{x_{\tau i},y_{\tau i}\}\right).
\]
For every fixed realization $x$ of $\mathcal A$, this is the minimum
Hamming distance from $x$ to a realization of $\mathcal C$.
\end{lemma}
\begin{proof}
The labels on $U$ are fixed. In $V_\tau$, at most
$\min\{x_{\tau i},y_{\tau i}\}$ vertices can receive label $i$ in
both maps. These upper bounds can be achieved simultaneously: for
each $i$, match that many occurrences of $i$ in $y$ to occurrences
of $i$ in the fixed map $x$, and then biject the remaining positions
arbitrarily. Such bijections can be chosen independently in each
type. The displayed formula follows by subtracting the maximum
number of agreements from $n$.
\end{proof}

In particular, if a collection of permitted repairs is invariant
under $\Gamma$, the minimum repair distance depends only on the
description of the disturbed word. The minimum is computed by
minimizing $\delta$ over the descriptions of permitted repairs.

\subsection{Graphs with a small vertex cover}

\begin{proposition}\label{rec:cover}
Given a vertex cover $U$ of size $s$, the values $B(G)$ and $B_{\mathrm w}(G)$
can be computed in time $g(s)(n+1)^{O(1)}$, for a computable function $g$.
\end{proposition}
\begin{proof}
Put $I=V(G)\setminus U$. Since $I$ is independent,
$k=\chi(G)\leq s+1$. We first compute $k$. For each
$j=1,\ldots,s+1$, enumerate the proper maps $\varphi:U\to[j]$.
Such a map extends to a proper $j$-coloring of $G$ if and only if
each nonempty type $\tau$ has a color outside $\varphi(\tau)$.
Indeed, all vertices of a type can be assigned any one such color.
The first successful value of $j$ is $k$.

Direct enumeration of the type counts would give an exponent depending
on $s$. Instead, we express the threshold condition by a short formula
in Presburger arithmetic: integer linear inequalities and congruences
with fixed coefficients, combined by Boolean operations and quantifiers.
Effective quantifier elimination converts such a formula into an
equivalent quantifier-free one. We use this established result, not
its quantitative refinements; see Haase et al.~\cite[Corollary~3.2]{haase2024}.
Universal quantifiers can be eliminated by negation.

Include all $2^s$ types, including empty ones. Introduce free variables
$\mathbf N=(N_\tau)_{\tau\subseteq U}$ for their sizes and a free
threshold variable $\ell$. These variables are not yet assigned their
input values. For an integer array $\mathbf x=(x_{\tau i})$ and a
map $a:U\to[k]$, use the following predicates and linear forms:
\[
\begin{aligned}
R(\mathbf x,\mathbf N)
&:\quad x_{\tau i}\geq0\ \text{for all }\tau,i,\qquad
          \sum_i x_{\tau i}=N_\tau\ \text{for all }\tau,\\
P_a(\mathbf x)
&:\quad a\text{ is proper on }G[U],\qquad
          x_{\tau i}=0\ \text{whenever }i\in a(\tau),\\
\mu_i(a,\mathbf x)&=|a^{-1}(i)|+\sum_\tau x_{\tau i}.
\end{aligned}
\]
Here properness of the fixed map $a$ is a Boolean constant.
The first predicate says that $\mathbf x$ describes a word; adding
the second says that the word is proper. Write
$\boldsymbol\mu(a,\mathbf x)=(\mu_i(a,\mathbf x))_{i=1}^k$.

Let $c,a:U\to[k]$ represent, respectively, an original coloring and
a disturbance. The predicate
\[
\begin{aligned}
A_{c,a}(\mathbf x,\mathbf z,\mathbf N):\quad&
R(\mathbf x,\mathbf N)\ \wedge\ P_c(\mathbf x)\
\wedge\ R(\mathbf z,\mathbf N)\\
&{}\wedge\
\boldsymbol\mu(c,\mathbf x)=\boldsymbol\mu(a,\mathbf z)
\end{aligned}
\]
says that the disturbed description $(a,\mathbf z)$ has a feasible
profile, witnessed by the proper description $(c,\mathbf x)$.
For a target map $b:U\to[k]$, put
$d_U(a,b)=|\{u\in U:a(u)\ne b(u)\}|$. Define
$T_{a,b}(\mathbf z,\mathbf y,\mathbf h;\mathbf N,\ell)$ by the
conjunction of
\[
\begin{gathered}
R(\mathbf y,\mathbf N),\qquad P_b(\mathbf y),\qquad
\boldsymbol\mu(b,\mathbf y)=\boldsymbol\mu(a,\mathbf z),\\
0\leq h_{\tau i}\leq z_{\tau i},\qquad
h_{\tau i}\leq y_{\tau i}\quad\text{for every }\tau,i,\\
d_U(a,b)+\sum_\tau N_\tau-\sum_{\tau,i}h_{\tau i}\leq\ell.
\end{gathered}
\]
By \cref{rec:distance}, this predicate has witnesses
$\mathbf y,\mathbf h$ exactly when a same-profile proper target
with core map $b$ is within distance $\ell$. Indeed, the greatest
possible total agreement in each type is
$\sum_i\min\{z_{\tau i},y_{\tau i}\}$, and is attainable for every
fixed realization of the source. The variables $h_{\tau i}$
encode these agreements using only linear constraints.

Let $\Psi_a(\mathbf z;\mathbf N,\ell)$ abbreviate
$\bigvee_{b:U\to[k]}\exists\mathbf y\,\exists\mathbf h\,
T_{a,b}(\mathbf z,\mathbf y,\mathbf h;\mathbf N,\ell)$.
Consider the formula
\begin{equation}\label{rec:presburger}
\Phi(\mathbf N,\ell)=
\bigwedge_{c,a:U\to[k]}\ \forall\mathbf x\,\forall\mathbf z\
\left(
A_{c,a}(\mathbf x,\mathbf z,\mathbf N)
\ \Longrightarrow\
\Psi_a(\mathbf z;\mathbf N,\ell)
\right).
\end{equation}
All quantified arrays range over the integers. After substituting
$N_\tau=n_\tau$, this formula holds precisely when $B(G)\leq\ell$.
Every admissible disturbance has at least one proper profile witness,
so the implication requires a repair for it. A disturbance with an
infeasible profile has no such witness and imposes no condition.
Although the witness is universally quantified, the permitted repairs
depend only on the disturbance, not on which witness was used.
For weak villainy, delete the profile equality from $T_{a,b}$ only;
the source-feasibility predicate $A_{c,a}$ is unchanged. The resulting
formula is equivalent to $B_{\mathrm w}(G)\leq\ell$.

Both formulas have length bounded by a function of $s$: there are
$k\leq s+1$ colors, $2^s$ types, and at most $k^{2s}$ pairs of
core maps. All coefficients and core constants are bounded in $s$.
Crucially, $\mathbf N$ and $\ell$ remain free during quantifier
elimination. Thus the elimination time, the output length, and its
coefficient lengths are bounded by a computable function of $s$,
independently of the numerical type sizes.
Substituting their actual values and evaluating the resulting formula
takes $g(s)(\log(n+1))^{O(1)}$ time. This is a uniform algorithm:
the formulas are generated from the given graph on $U$ and then
processed by an effective elimination procedure. Binary search on
$0\leq\ell\leq n$, together with input processing, computes both
values in the claimed time.
\end{proof}

\subsection{Graphs with a small complement cover}

Suppose next that $W=V(G)\setminus U$ is a clique. The chromatic
number need not be bounded, but its computation still depends on
only a bounded part of the graph.

\Needspace{6\baselineskip}
\begin{lemma}\label{rec:chromatic}
If $|U|=s$ and $V(G)\setminus U$ is a clique, then $\chi(G)$ can be
computed in time $2^{O(s\log(s+2))}(n+1)^{O(1)}$.
\end{lemma}
\begin{proof}
Enumerate the partitions $\pi$ of $U$ into independent sets. In any
proper coloring inducing $\pi$, a vertex of $W$ either has a color
not used on $U$, or shares a color with one part of $\pi$.
No two vertices of $W$ can share a color. Thus the possible sharing
is described by a matching in the bipartite graph whose sides are
$\pi$ and $W$, with a part adjacent to $v\in W$ when $v$ has no
neighbor in that part. If a maximum matching has size $a(\pi)$,
the fewest colors for this partition are
\[
|\pi|+|W|-a(\pi).
\]
Taking the minimum over $\pi$ gives $\chi(G)$. There are at most
$s^s$ partitions for $s\geq1$, and each matching is computed in
polynomial time. The empty partition handles $s=0$.
\end{proof}

We shall use the following elementary fact to handle the many
singleton color classes. It does not depend on the structure of $G$.

\begin{lemma}\label{rec:singleton}
Let $\mathbf m,\mathbf m'$ be positive profiles on $[k]$ of the
same total size, and suppose that $G$ admits a proper coloring of
profile $\mathbf m'$. Let $J\subseteq[k]$ contain every color having
multiplicity at least two in either profile. Replace each color
outside $J$ by a common label $*$. For a fixed word $w$ of profile
$\mathbf m$, the minimum Hamming distance to a proper coloring of
profile $\mathbf m'$ equals the minimum Hamming distance between
the resulting label maps.
\end{lemma}
\begin{proof}
Different labels imply different colors, so passing to labels can
only decrease a distance. To prove the reverse inequality, take any
proper coloring $y$ of profile $\mathbf m'$ and let $R_w,R_y$ be
the positions with label $*$ in the two maps. Both restrictions
to the star positions are bijections onto the same color set
$S=[k]\setminus J$, since every color in $S$ has multiplicity one
in both profiles.

Keep the colors of $y$ outside $R_y$. On $R_w\cap R_y$, assign the
colors of $w$. Fill $R_y\setminus R_w$ bijectively with the unused
colors in $S$. This preserves the target profile and properness:
all changed color classes are singletons. It also makes the two
colorings agree at every position at which their label maps agree.
Consequently its Hamming distance from $w$ is exactly the label
distance of the original pair. Minimizing proves the assertion.
\end{proof}

We can now combine the type description with \cref{rec:singleton}:
only the colors occurring more than once need individual names.
The resulting count is controlled by the size of the cover and by
the difference between the order and chromatic number.

\begin{proposition}\label{rec:cocover}
Suppose that $|U|\leq s$, $V(G)\setminus U$ is a clique, and
$n-\chi(G)\leq e$. Then both $B(G)$ and $B_{\mathrm w}(G)$ can
be computed in time $f(s,e)(n+1)^{O(1)}$, for a computable function $f$.
One may take
\[
f(s,e)=2^{O(se+(s+e)\log(s+e+2))}.
\]
\end{proposition}
\begin{proof}
Compute $k=\chi(G)$ by \cref{rec:chromatic}, and put
$d=n-k\leq e$. For every positive profile $\mathbf m$ on $[k]$,
\[
\sum_i(m_i-1)=d.
\]
Consequently its set $N$ of nonsingleton colors has size at most
$d$, and
\[
\sum_{i\in N}m_i=d+|N|\leq2e.
\]
Up to a permutation of color names, there are only a bounded number
of such profiles: their nonsingleton excesses $m_i-1$ form a
partition of $d$. A \emph{profile shape} is a profile considered up
to a permutation of color names. We enumerate these shapes, choose a
labeled representative of each, and discard it if it is not the
profile of a proper coloring. The properness test is included in
the description enumeration below. A simultaneous permutation of
color names preserves repair distances, so these representatives
suffice for the outer maximum in both parameters.

First consider $B(G)$ and fix one such profile $\mathbf m$.
Use labels $N\cup\{*\}$, with $*$ replacing singleton colors.
The nonsingleton labels occupy at most $2e$ vertices. For a label
map to come from a proper coloring, the vertices assigned each
nonstar label must form an independent set. This can be checked from
the description: its vertices in $U$ must be independent, the label
may occur at most once in the clique $W$, and any such vertex must have a type
disjoint from that label's vertex set in $U$. The star positions impose no
constraint, because they receive distinct singleton colors.
We retain only descriptions having total count $m_i$ for every
$i\in N$ and total star count $k-|N|$, for both sources and targets.

By \cref{rec:singleton}, the distance to a best repair of the
same profile is exactly the distance to a best proper label map.
By \cref{rec:distance}, this minimum, and then the maximum
over disturbed words, can be evaluated using descriptions alone.
The number of descriptions is bounded by a function of $s,e$.
For example, after choosing the at most $(e+1)^s$ labels on $U$,
the at most $2e$ nonstar occurrences outside $U$ are distributed
among at most $e2^s$ label--type pairs. An overcount is
\[
(e+1)^s(e2^s+1)^{2e}.
\]
Star counts are then determined by the type sizes. Testing all
descriptions and comparing all pairs proves the first running-time
bound.

For weak villainy, the target may have a different profile. Here
the choice of target nonsingleton colors must be made after the
disturbed word has been fixed. We describe this order explicitly.
For each source description using labels $N\cup\{*\}$, choose
one full-color representative $w$, assigning the source singleton
colors bijectively to its star positions. All words represented
by this reduced description have the same weak repair distance.
Indeed, two such words are related by an automorphism in $\Gamma$
followed by a permutation of singleton color names that fixes
$N$. Both transformations preserve Hamming distances and the set
of proper $k$-colorings.

With $w$ fixed, let $N'$ be the nonsingleton colors of a target.
Colors in $N'\setminus N$ are singleton in the source; call them
\emph{promoted colors}. Each is identified by its unique position
in $w$. Enumerating their names separately would introduce a factor
depending on $n$ with exponent $e$. We avoid this by using symmetries
that fix the full source word.

Let $\gamma\in\Gamma$ permute only source-star positions within
each outside type, fixing all other vertices. Define a color
permutation $\pi_\gamma$ by
$\pi_\gamma(w(v))=w(\gamma(v))$ on source singleton colors, and
fix every color in $N$. For any target coloring $y$, put
\[
(T_\gamma y)(v)=\pi_\gamma\bigl(y(\gamma^{-1}(v))\bigr).
\]
Then $T_\gamma w=w$. Moreover, $T_\gamma$ preserves properness
and Hamming distance, because $\gamma$ is a graph automorphism
and $\pi_\gamma$ is a permutation of colors. In particular,
$\dH(w,T_\gamma y)=\dH(w,y)$.

These transformations act transitively on choices of promoted colors
that occupy the same source singleton positions in $U$ and have the
same number of source positions in each outside type. Consequently,
we need only choose a subset of the source singleton positions in
$U$, and integers $h_\tau\geq0$ with $\sum_\tau h_\tau\leq e$,
not exceeding the available number of source-star positions of type
$\tau$. For each choice, fix one set of $h_\tau$ singleton
names from that type. There are at most
$2^s(2^s+1)^e$ such choices: the count vectors are overcounted
by sequences of $e$ type names or unused slots.
Choose also the subset $N'\cap N$, and retain only sets with
$|N'|\leq d$. For each resulting $N'$, enumerate every assignment
\[
m'_i\geq2\ (i\in N'),\qquad m'_i=1\ (i\notin N'),\qquad
\sum_{i\in N'}(m'_i-1)=d.
\]
Every target is equivalent, under a transformation fixing $w$,
to one of these choices. Its excesses are carried along by the
color permutation and occur in the enumeration. For each choice, use the
label set $J\cup\{*\}$ with $J=N\cup N'$. The star label now
represents exactly the colors that are singleton in both profiles.
\cref{rec:singleton} applies to this fixed source and target
profile, and \cref{rec:distance} computes the minimum over
proper target descriptions. Colors in $N'\setminus N$ have one
source occurrence, and colors in $N\setminus N'$ have one target
occurrence; these occurrences retain their individual labels.
Thus no potential agreement is lost by the reduction. Each target
description must have total count $m'_i$ for $i\in J$ and exactly
$k-|J|$ star positions; target profiles with no proper description
are discarded.

There are at most $2^{s+e}(2^s+1)^e$ choices of $N'$, and a number
depending only on $e$ of assignments of its positive excesses.
Each map on the union label set has at most $3e$ nonstar
occurrences and at most $2e$ nonstar labels. Its descriptions are
therefore bounded in number by
\[
(2e+1)^s(2e2^s+1)^{3e}.
\]
Minimize over the target profiles and descriptions for this fixed
$w$, then maximize over source descriptions and admissible source
profile shapes. The order of these operations is exactly the
order in \cref{def:repair-radii}. This proves the second
running-time bound. The displayed description and orbit bounds, the number
of profile shapes, and \cref{rec:chromatic} give the stated
choice of $f(s,e)$, with the cases $e=0$ interpreted directly.
\end{proof}

\subsection{The recognition algorithm}

\begin{proof}[Proof of \cref{rec:main}]
Fix $t$. If $G$ is empty, return zero for both parameters.
Otherwise, compute a maximum matching in $G$. If its size is at most
$t$, its endpoints form a vertex cover of size at most $2t$.
\cref{rec:cover} computes both parameters exactly;
truncate each answer at $t+1$.

Otherwise $\nu(G)>t$. Compute a maximum matching in $\overline G$.
If $\nu(\overline G)>2t$, then
$\chi(G)\leq n-\nu(\overline G)<n-2t$.
\cref{rec:dichotomy} implies $B_{\mathrm w}(G)>t$, and hence also
$B(G)>t$. Return $t+1$ for both parameters.

In the remaining case, the endpoints of a maximum matching in
$\overline G$ give a set $U$ of size at most $4t$ whose complement
is a clique in $G$. Compute $\chi(G)$ using
\cref{rec:chromatic}. If $\chi(G)<n-2t$, the same dichotomy
again gives both answers as $t+1$. If $\chi(G)\geq n-2t$, apply
\cref{rec:cocover} with $s\leq4t$ and $e\leq2t$, and
truncate the exact answers at $t+1$.

This exhausts all graphs and proves correctness. By \cref{rec:cover},
the first branch takes time $g(2t)(n+1)^{O(1)}$, after replacing
$g$ by a nondecreasing computable bound if necessary.
By \cref{rec:cocover}, the complement-cover branch takes time
$2^{O((t+1)^2)}(n+1)^{O(1)}$. Input processing and matching
computations add only fixed polynomial factors. Their exponents
are independent of $t$, proving the stated bound with a computable
$f$. The argument includes $t=0$.

To obtain the exact values with weak villainy as parameter, run this
structural algorithm at
thresholds $0,1,2,4,\ldots$. Whenever an exact-computation branch
is reached, retain both answers before truncation and stop.
Continue only after a rejection, which certifies
$B_{\mathrm w}(G)>t$. By \cref{rec:dichotomy}, rejection is
impossible once $t\ge B_{\mathrm w}(G)$.
If $B_{\mathrm w}(G)\ge1$, the largest tested threshold is at most
$2B_{\mathrm w}(G)$; if it is zero, the first run succeeds.
Taking a nondecreasing envelope of the parameter bounds and
absorbing the number of runs proves the claim.
\end{proof}

\begin{corollary}\label{rec:fixed-complexity}
Unless $\mathrm P=\mathrm{NP}$, there is no fixed integer $t$ for
which recognition of $B(G)=t$ is NP-complete. The same conclusion
holds for $B_{\mathrm w}(G)=t$.
\end{corollary}

\section{Concluding remarks and questions}\label{sec:questions}

The composition laws quantify how global color counts couple components.
For repeated copies, \cref{union:repetition-stabilization,union:ordinary-repetition}
give effective eventual formulas for both repair rules. Their slopes
also decide whether the difference between the two values stays bounded.
The next questions ask for structural descriptions and sharper bounds
than the finite local optimizations provide.

\begin{problem}\label{problem:union-equality}
Find structural or efficiently checkable versions of the criteria in
\cref{union:weak-equality-general,union:ordinary-equality}.
In particular, characterize the graphs satisfying
$B_{\mathrm w}(G)=F_{\chi(G)}(G)$, equivalently those for which weak
villainy is additive over every repetition.
\end{problem}

For either parameter $F$, the \emph{superadditivity deficit} is
$F(G)+F(H)-F(G\sqcup H)$; a positive value measures a failure of
superadditivity.

\Needspace{6\baselineskip}
\begin{problem}\label{problem:union-deficits}
Determine sharp bounds on the superadditivity deficits at prescribed
orders and chromatic numbers. For weak villainy, the unresolved deficit
case has unequal chromatic numbers.
Give a structural characterization of the graphs satisfying
$\Lambda(G)=F_{\chi(G)}(G)$, the finite optimization criterion for
bounded difference in \cref{union:repetition-gap}.
\end{problem}

The explicit stabilization bound in \cref{union:repetition-stabilization}
comes from a general zero-sum estimate and need not be sharp.
\begin{problem}\label{problem:repetition-onset}
Determine sharp bounds on the least stabilization index for weak
repetition, meaning the least index from which the eventual periodic
formula holds, and give graph-theoretic descriptions of its least period.
Obtain comparable explicit bounds for the ordinary recurrence in
\cref{union:ordinary-repetition}.
\end{problem}

The odd-cycle theorem of \cite[Theorem~3.15]{jiang2026exact} already
gives singleton response families on a three-color palette, as discussed
after \cref{union:ordinary-composition}.
\begin{problem}\label{problem:response-compression}
For a fixed palette size, determine further graph classes whose response
families in \cref{union:ordinary-composition} admit polynomial-size
representations supporting polynomial-time composition and queries.
How much can these families be compressed without changing ordinary
villainy after any further disjoint union?
\end{problem}

The parameter dependence in \cref{rec:main} comes from general-purpose
quantifier elimination. Can both values be computed by a direct
combinatorial algorithm in time
$2^{B_{\mathrm w}(G)^{O(1)}}(|V(G)|+1)^{O(1)}$?

\section*{Acknowledgments}

Lujia Wang is partially supported by the National Natural Science
Foundation of China (grant No.~12371359). Jiang Jiang, Yuhan Jiang,
Yiyang Li, and Aitao Zhang are
supported by the Math+AI research project of Hainan Bielefeld University
of Applied Sciences.

\section*{Contributions and use of AI}

The authors used AI tools from OpenAI and Anthropic to assist with developing and
checking arguments, computational experiments, literature searches,
and manuscript preparation.
The authors are responsible for all proofs, computations, and conclusions.


\end{document}